\documentclass[12pt,a4paper]{article}
\usepackage{amssymb}
\usepackage{amsmath}
\usepackage{amsthm}
\usepackage[T1,OT1]{fontenc}
\usepackage{graphicx}
\usepackage[hidelinks]{hyperref}
\usepackage{calrsfs}
\usepackage{dsfont}
\usepackage{bbm}

\newtheorem{theorem}{Theorem}[section]

\newtheorem{lemma}[theorem]{Lemma}

\theoremstyle{definition}

\newtheorem*{note}{Note}

\newcommand{\E}{\mathbb{E}}
\newcommand{\Z}{\mathbb{Z}}

\begin{document}

\title{The Heesch number in $\mathbb{E}^3$ is unbounded}
\author{Bojan Ba\v si\'c\\
\em \small Department of Mathematics and Informatics, University of Novi Sad,\\
\em \small Trg Dositeja Obradovi\'ca 4, 21000 Novi Sad, Serbia\\
\small \tt bojan.basic@dmi.uns.ac.rs
}
\date{}

\maketitle

\begin{abstract}
The Heesch number of a figure is the largest number of times that the figure can be surrounded by layers of its congruent copies. Heesch's problem asks whether finite Heesch numbers are bounded. We show that in the three-dimensional Euclidean space, and thus also in higher-dimensional Euclidean spaces, they are not: for every positive integer $h$, we construct a shape whose Heesch number is finite and at least $h$. The construction is explicit, and it starts from a four-point set of integers that does not tile the line, but packs long intervals.

\emph{Mathematics Subject Classification (2020):} 52C22, 05B45, 52C17, 05B50, 51M20

\emph{Keywords:} Heesch number; Heesch's problem; tiling; tessellation; corona; polycube
\end{abstract}

\section{Introduction}

Let $T$ be a figure in the plane that does not tile the plane by congruent copies. It may nevertheless be possible to surround $T$ by its copies, so that the surrounding copies completely ``tuck in'' $T$, without any gaps nor overlaps; such a layer of copies is called a \emph{corona} of $T$. It may then be possible to surround the obtained form by a further corona, and so on. The largest number of coronas that can be built in this way is the \emph{Heesch number} of $T$, named after Heinrich Heesch \cite{heesch68}, who gave an example of a figure that can be surrounded once, but not twice (an even earlier example can be found in a book on mathematical recreations by Lietzmann \cite{lietzmann}, published in 1928). If $T$ tiles the plane, its Heesch number is infinite, and it can be shown that the converse is also true; thus the Heesch number measures, in a sense, how close a figure that does not tile the plane comes to tiling it. The (still open) question of whether the set of finite Heesch numbers is bounded above is referred to as \emph{Heesch's problem.}

Progress on Heesch's problem has been slow. When Gr\"unbaum and Shephard wrote their monograph \cite{gs}, no figure with finite Heesch number larger than $1$ was known to them. Fontaine \cite{fontaine} found an infinite family of figures with Heesch number $2$; Ammann \cite{ammann} found a notched hexagon with Heesch number $3$; and Mann \cite{mann04} constructed infinite families of figures with Heesch number $5$. A figure with Heesch number $6$ was found in \cite{basic21}, and $6$ is still the largest known finite Heesch number in the plane. Extensive computer searches have not changed this picture: Mann and Thomas \cite{mann-thomas} examined polyforms with marked edges, and Kaplan \cite{kaplan}, with the help of a SAT solver, computed the Heesch numbers of all non-tiling polyominoes, polyhexes and polyiamonds up to a considerable size; the largest value among them turned out to be $4$.

In some related settings, on the other hand, finite Heesch numbers are known to be unbounded. If, instead of a single figure, a set of two or more figures is allowed, then the corresponding Heesch number is unbounded \cite{basic15,bds} (for a suitable definition of the notion of Heesch number for sets of more figures). In the hyperbolic plane, Tarasov \cite{tarasov} found polygons with arbitrarily large finite Heesch numbers (and see the recent preprint \cite{maiti} for convex examples). In higher dimensions, it was shown in \cite{bs22} (for dimensions that are powers of $2$) and in \cite{bs25} (for all $d$, $d\geqslant2$) that there exist figures in $\E^d$ with Heesch number $d-1$; thus finite Heesch numbers are unbounded if the dimension is allowed to grow, but this says nothing about any fixed dimension.

Part of the interest in Heesch's problem comes from its connection with decidability. Suppose that, for some class of figures in which two copies of a figure can be placed next to each other only in finitely many essentially different ways, finite Heesch numbers were bounded. Then it would be decidable whether a given figure from the class tiles the plane: it would suffice to check whether the figure can be surrounded one more time than the bound allows. It is not known whether the tiling problem for a single figure in the plane is decidable. For finite sets of figures, it is undecidable, by the classical results of Berger \cite{berger} and Robinson \cite{robinson} on Wang's domino problem \cite{wang}. For translations of a single tile, Greenfeld and Tao \cite{gt} proved undecidability when the dimension is a part of the input (see \cite{gt} for the precise statement), and a very recent preprint of Demaine and Langerman \cite{dl} announces that it is undecidable whether a single polycube tiles $\E^3$ by translations. If the latter is the case, then an argument along the same lines as above shows that finite Heesch numbers of polycubes in $\E^3$, when only translations are allowed, are unbounded. Such an argument is, however, nonconstructive, and it says nothing about what the polycubes in question look like, or how large they are.

In the present article, we show, by an explicit and elementary construction, that finite Heesch numbers in $\E^3$ are unbounded, all isometries being allowed. Namely, for every positive integer $h$, we construct a shape in $\E^3$ (homeomorphic to a ball) whose Heesch number is finite and at least $h$. The polyhedron is obtained from a polycube consisting of $48h+64$ unit cubes by adding small bumps and nicks to its faces. Note that the result extends to $\E^d$ for every $d$ with $d\geqslant 3$, simply by extending the shape along the additional coordinate axes and adding bumps and nicks to ensure unique locking.

This would not be the first time that a question on tilings is answered in space before its planar version is settled. Hilbert's eighteenth problem \cite{hilbert} asked, among other things, whether there exists a polyhedron that tiles space but is not a fundamental domain of any group of isometries; such a polyhedron was found by Reinhardt \cite{reinhardt} in 1928, while the first planar example was given in 1935 by Heesch himself \cite{heesch35}. Similarly, Schmitt constructed a nonconvex (weakly) aperiodic polyhedron in an unpublished work in 1988 (see \cite{danzer}, where the construction was subsequently refined to obtain a convex, face-to-face example: the Schmitt–Conway–Danzer biprism; though for the latter, reflections must be forbidden); three and a half decades passed before Smith, Myers, Kaplan and Goodman-Strauss finally discovered aperiodic figures in the plane \cite{smkg1,smkg2}.

The idea of our construction is simple. On the line, the set $\{0,2n-1,2n+1,4n\}$ does not tile $\Z$ by translations, but it can pack (without gaps) an interval of length about $4n^2$, which is much larger than its diameter $4n$. We place four copies of a suitable bent path of cubes at these four points along a diagonal direction of the cubic lattice, so that they form a connected polycube (see Figure \ref{fig}, where the case $n=16$ is shown). The path is chosen so that, on the one hand, the polycube does not tile $\Z^3$, and on the other hand, the long one-dimensional packing lifts to a packing of a large cube by translates of the polycube. Finally, bumps and nicks prevent the copies from being rotated or reflected, and the large packed cube provides many coronas.

Let us now briefly describe the content of the article. In Section \ref{s2} we give the precise definitions and recall some basic facts. Section \ref{s3} is devoted to the one-dimensional set mentioned above, and in Section \ref{s4} we construct the polycube and show that it packs large cubes. In Section \ref{s5} we add the bumps and nicks and prove the main result.

\section{Preliminaries}\label{s2}

By a \emph{shape} in $\E^d$ we mean a compact set that is the closure of its interior. A \emph{copy} of a shape $T$ is any set congruent to $T$, where all isometries of $\E^d$ (including reflections) are allowed. A \emph{packing} by copies of $T$ is a collection of copies of $T$ with pairwise disjoint interiors, and a \emph{tessellation} of $\E^d$ is a packing such that the union of all of its members is entire $\E^d$; if such a tessellation exists, we say that $T$ \emph{tessellates} $\E^d$. (For $d=2$, the term \emph{tiling} is rather more common, but in higher dimensions we shall stick to \emph{tessellation.})

Let $n$ be a nonnegative integer. We say that a shape $T$ can be \emph{surrounded $n$ times} if and only if there exist finite collections of shapes $\mathcal C_1,\mathcal C_2,\dots,\mathcal C_n$ such that:
\begin{itemize}
\item for each $i$, $1\leqslant i\leqslant n$, each shape from $\mathcal C_i$ is a copy of $T$;
\item every two different shapes from $\{T\}\cup\bigcup_{i=1}^n\mathcal C_i$ have disjoint interiors;
\item for each $i$, $1\leqslant i\leqslant n$, each shape from $\mathcal C_i$ has a common boundary point with some shape from $\mathcal C_{i-1}$ (where, by convention, we let $\mathcal C_0=\{T\}$);
\item for each $i$, $1\leqslant i\leqslant n$, $\bigcup\left(\bigcup_{j=0}^i{\mathcal C_j}\right)$ contains $\bigcup\left(\bigcup_{j=0}^{i-1}{\mathcal C_j}\right)$ completely in its interior.
\end{itemize}

The collection $\mathcal C_i$ is called the \emph{$i^{\text{th}}$ corona}. The \emph{Heesch number} $H(T)$ of $T$  is the maximal nonnegative integer $n$ such that $T$ can be surrounded $n$ times. If such a maximum does not exist, then we define the Heesch number to be infinite. We have that $H(T)<\infty$ if and only if $T$ does not tessellate $\E^d$: the direction $(\Rightarrow)$ is trivial, while the direction $(\Leftarrow)$ follows from the Extension Theorem \cite[Theorem 3.8.1]{gs} (whose proof for the plane works in any dimension).

We shall also consider packings and tessellations in the integer lattice $\Z^d$. A family of translates of a finite subset of $\Z^d$ is a \emph{packing} if the translates are pairwise disjoint, and a \emph{tessellation} if they form a partition of $\Z^d$. In $\Z^3$, we identify a point $x$ with the unit cube $x+[0,1]^3$; in this way, a finite subset $K$ of $\Z^3$ becomes a polycube (the union of the corresponding cubes), and packings of $\Z^3$ by translates of $K$ become packings of $\E^3$ by integer translates of this polycube.

\section{Four points on a line}\label{s3}

In this section we study a four-point set of integers that does not tessellate $\Z$, but comes very close to tessellating it. For a positive integer $n$, let $a=2n-1$ and $b=2n+1$, and let $F_n=\{0,a,b,a+b\}$; explicitly, $F_n=\{0,2n-1,2n+1,4n\}$. The gaps between consecutive elements of $F_n$ are $2n-1$, $2$ and $2n-1$.

\begin{lemma}\label{l31}
For every positive integer $n$, the set $F_n$ does not tessellate $\Z$.
\end{lemma}

\begin{proof}
Suppose the contrary, and let $C$ be the set of translation vectors in a tessellation of $\Z$ by translates of $F_n$. Let $\mathds{1}(t)$ be defined by $\mathds{1}(t)=1$ if $t\in C$, and $\mathds{1}(t)=0$ otherwise. Then
\begin{equation}\label{e1}
\mathds{1}(t)+\mathds{1}(t-a)+\mathds{1}(t-b)+\mathds{1}(t-a-b)=1
\end{equation}
for every integer $t$. Put $f(t)=\mathds{1}(t)+\mathds{1}(t-a)$. By \eqref{e1} we have $f(t)+f(t-b)=1$, so the values of $f$ alternate between $0$ and $1$ along every arithmetic progression with difference $b$; as $a$ is odd, this gives $f(t)+f(t-ab)=1$. Now put $g(t)=\mathds{1}(t)+\mathds{1}(t-ab)$. The second-to-last equality reads $g(t)+g(t-a)=1$, so in the same way, as $b$ is odd, we get $g(t)+g(t-ab)=1$, that is,
\[
\mathds{1}(t)+2\cdot\mathds{1}(t-ab)+\mathds{1}(t-2ab)=1
\]
for every integer $t$. Since $\mathds{1}$ takes only the values $0$ and $1$, this gives $\mathds{1}(t-ab)=0$ for every $t$, that is, $C=\varnothing$, a contradiction.
\end{proof}

\begin{note}
Lemma \ref{l31} is also a special case of a theorem of Newman \cite{newman}: a set of $p^\alpha$ integers, where $p$ is a prime, tessellates $\Z$ if and only if, among the largest powers of $p$ dividing the differences of two elements of the set, there are at most $\alpha$ distinct ones. For $F_n$, we have $p^\alpha=4$, while the largest powers of $2$ dividing the differences $2n-1$, $2$ and $4n$ are three distinct numbers.
\end{note}

On the other hand, $F_n$ can pack intervals whose length is of order $n^2$, while its own diameter is only $4n$. This discrepancy is the source of the large Heesch numbers in Section \ref{s5}.

\begin{lemma}\label{l32}
Let $n$ be an integer, $n\geqslant4$. Then there is a finite packing of $\Z$ by translates of $F_n$ that covers $4n^2-10n-2$ consecutive integers.
\end{lemma}

\begin{proof}
Since $a$ and $2b$ are coprime, there is an integer $\gamma$ such that $a\gamma\equiv1\pmod{2b}$. Note that $\gamma$ is odd. For an integer $t$, let $v(t)=1$ if the remainder of $\gamma t$ modulo $2b$ is less than $b$, and $v(t)=0$ otherwise. As $\gamma b\equiv b\pmod{2b}$ (recall that $\gamma$ is odd), we have
\begin{equation}\label{e2}
v(t)+v(t-b)=1
\end{equation}
for every integer $t$. Moreover, as $\gamma a\equiv1\pmod{2b}$, along every arithmetic progression with difference $a$ the values of $v$ repeat the pattern of $b$ ones followed by $b$ zeros. The integers from $0$ to $ab-1$ split into $a$ progressions with difference $a$, each having $b$ terms, and along each of them $v$ changes its value at most once. Therefore, we can assign to each integer $t$ with $-a\leqslant t\leqslant ab-1$ a value $u(t)$, equal to $0$ or $1$, so that
\begin{equation}\label{e3}
u(t)+u(t-a)=v(t)\quad\text{whenever $0\leqslant t\leqslant ab-1$}
\end{equation}
(along each progression, the zeros of $v$ force zeros of $u$, while on a block of ones $u$ simply alternates between $0$ and $1$; as $v$ changes its value at most once, there is no conflict). Adding \eqref{e3} for $t$ and for $t-b$, and using \eqref{e2}, we obtain
\begin{equation}\label{e4}
u(t)+u(t-a)+u(t-b)+u(t-a-b)=1\quad\text{whenever $b\leqslant t\leqslant ab-1$}.
\end{equation}
Let $C$ be the set of those integers $t$ for which $u(t)=1$ and $b\leqslant t\leqslant ab-1-a-b$. Each translate $F_n+t$ with $t\in C$ is contained in the interval $[b,ab-1]$, and a point of this interval belongs to $F_n+t$ if and only if $t$ is one of the four numbers appearing on the left-hand side of \eqref{e4}; hence, by \eqref{e4}, these translates form a packing. Moreover, if $a+2b\leqslant t\leqslant ab-1-a-b$, then all the four numbers $t$, $t-a$, $t-b$, $t-a-b$ lie between $b$ and $ab-1-a-b$, and \eqref{e4} shows that $t$ is covered. It remains to note that the number of integers between $a+2b$ and $ab-1-a-b$ (inclusive) is $ab-2a-3b$, which equals $4n^2-10n-2$.
\end{proof}

\section{From four points to a polycube}\label{s4}

Let $e=(1,1,0)$, and for a positive integer $n$, let $P_n=\{te:t\in F_n\}$; thus $P_n$ consists of the points $0$, $(2n-1)e$, $(2n+1)e$ and $4ne$ of $\Z^3$. Every translate of $P_n$ lies on a line parallel to $e$, so a tessellation of $\Z^3$ by translates of $P_n$ would induce a tessellation of each such line (which is a copy of $\Z$) by translates of $F_n$. Hence, by Lemma \ref{l31}, $P_n$ does not tessellate $\Z^3$. Of course, $P_n$ is very far from being connected. In order to connect its four points, we attach to each of them the same bent path of cubes.

Let $S_n$ be the polycube consisting of the $2n$ cubes $(r,0,0)$, where $r=0,1,\dots,2n-1$, and the $2n$ cubes $(2n-1,s,1)$, where $s=0,1,\dots,2n-1$. In other words, $S_n$ starts at the origin, runs through $2n$ cubes along the $x$-axis, then climbs one level up and runs through $2n$ further cubes in the direction of the $y$-axis; we call these two parts the \emph{lower} and the \emph{upper arm} of $S_n$. Let
\[
K_n=\bigcup_{p\in P_n}(S_n+p).
\]
The polycube $K_n$ thus consists of the four translates of $S_n$ by the vectors $0$, $(2n-1)e$, $(2n+1)e$ and $4ne$, which are, respectively, the blue, the brown, the green and the purple path in Figure \ref{fig}, where $K_n$ is shown for $n=16$. In this picture, the lower arms go down to the right and the upper arms go up to the right, the four points of $P_n$ are the first cubes of the four paths, and the paths follow each other in the direction of $e$.

\begin{figure}[htb]
\begin{center}
\includegraphics[width=\textwidth]{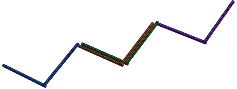}
\end{center}
\caption{The polycube $K_n$ for $n=16$.}\label{fig}
\end{figure}

The properties of $K_n$ that we need are the following ones.

\begin{itemize}
\item[(a)] Let the \emph{lower strip} be the union of the $2n$ lines parallel to $e$ that pass through the cubes of the lower arm of $S_n$, and let the \emph{upper strip} be the union of the $2n$ lines lying directly above them, in the plane $z=1$. Each of the two arms of $S_n$ has exactly one cube on each line of the corresponding strip. Therefore, the lower arms of the translates $S_n+te$, where $t\in\Z$, fill the lower strip without overlapping, and their upper arms fill the upper strip without overlapping; in particular, these translates are pairwise disjoint.
\item[(b)]By (a), the four paths forming $K_n$ are pairwise disjoint; thus, $K_n$ consists of $16n$ cubes.
\item[(c)] $K_n$ is connected. Indeed, as can be seen in Figure \ref{fig}, the last cube of the blue path lies directly above the first cube of the brown path (the endpoint $(2n-1,2n-1,1)$ of $S_n$ lies above the point $(2n-1)e$), the same holds for the green and the purple path, and the upper arm of the brown path passes over the lower arm of the green path, these two paths being translates of each other by $2e$. Note also that $K_n$ is contained in the box $[0,6n]\times[0,6n]\times[0,2]$.
\item[(d)]$K_n$ does not tessellate $\Z^3$. Indeed, suppose that some translates of $K_n$ form a tessellation of $\Z^3$. By (b), each of these translates is a disjoint union of translates of $P_n$. Consequently, these translates of $P_n$ also form a tessellation of $\Z^3$, which we have seen to be impossible.
\end{itemize}

The following lemma shows that $K_n$, just as $F_n$, comes very close to tessellating.

\begin{lemma}
    \label{p41}
Let $n$ be an integer, $n\geqslant4$. Then there is a finite packing by integers translates of $K_n$ that covers a cube of edge length $4n^2-12n-1$.
\end{lemma}

\begin{proof}
We construct the packing in three steps.

First, we move $K_n$ along $e$ as dictated by Lemma \ref{l32}: for each translate $F_n+t$ in the packing from Lemma \ref{l32}, we take the translate $K_n+te$. Let us cut the lower strip into \emph{rows}, that is, into its intersections with lines parallel to the $x$-axis, and the upper strip into \emph{columns}, that is, into its intersections with lines parallel to the $y$-axis; each row and each column consists of $2n$ cubes. As can be seen in Figure \ref{fig}, $K_n$ consists of the four rows whose $y$-coordinates belong to $F_n$ (these are the lower arms) and the four columns whose $x$-coordinates belong to $F_n+2n-1$ (these are the upper arms), and the translation by $te$ adds $t$ to all these coordinates. Thus, in each of the two strips, our translates of $K_n$ simply reproduce the packing from Lemma \ref{l32}, with rows, respectively columns, in place of integers. In particular, they do not overlap, and they cover all the rows whose $y$-coordinates lie in an interval of $4n^2-10n-2$ consecutive integers, as well as all the columns whose $x$-coordinates lie in this interval shifted by $2n-1$. Since $x-y$ takes only the values $0,1,\dots,2n-1$ on the strips, all the cubes of both strips whose $y$-coordinates lie in the interval without its first $2n-1$ elements are covered; this remaining interval consists of $4n^2-12n-1$ integers. In this way, we get two long blocks of cubes, lying one above the other, in the planes $z=0$ and $z=1$.

Second, we translate the packing obtained in the first step by the vectors $(2nj,0,0)$, where $j$ runs through an interval of integers. Such a translation does not change the $y$-coordinates, and it moves each of the two strips sideways by $j$ times its width (in Figure \ref{fig}, the translation by $(2n,0,0)$ moves each lower arm along itself by its own length). As the packing obtained in the first step lies in the two strips, its translates do not overlap, and the translates of our two blocks lie side by side, without gaps. In each of the planes $z=0$ and $z=1$, they form a parallelogram of height $4n^2-12n-1$ (in the direction of the $y$-axis), with two sides parallel to $e$ and two sides parallel to the $x$-axis, the latter as long as we wish; if they are long enough, the parallelogram contains a square of side $4n^2-12n-1$.

Finally, we translate the packing obtained in the second step by the vectors $(0,0,2k)$, where $k$ runs through an interval of integers. These translates lie in pairwise disjoint pairs of consecutive horizontal planes, so they do not overlap, and they cover, in each of as many consecutive horizontal planes as we wish, the square from the second step. In particular, they cover a cube of edge length $4n^2-12n-1$, which was to be proved.
\end{proof}

\section{Bumps, nicks and coronas}\label{s5}

So far, only translations have been considered. However, rotations and reflections can be easily excluded by suitable geometric decorations (these arguments are quite standard in the area). Let $Q$ be the unit cube $[0,1]^3$, modified as follows: to each face we attach a small bump, and into the opposite face, at the same place, we carve the matching nick; the bumps on different faces are chosen to be different. The bumps and nicks are chosen in such a way that a bump fits only into its own nick, and only in one way. From now on, let $T_n$ be a shape obtained by replacing each constituent unit cube of $K_n$ by $Q$. Note that any two unit cubes that share a face now have their corresponding faces decorated by a matching bump and nick; this means that, basically, $T_n$ is obtained from the polycube $K_n$ by putting the bumps and nicks on those faces of its cubes that lie on its boundary. As bumps and nicks are chosen in such a way that each copy of $T_n$ in $\E^3$ forces all of its neighboring copies to be its translates by integer vectors, and as $K_n$ does not tessellate $\mathbb{E}^3$ by integer translations, we conclude that $T_n$ does not tessellate $\E^3$, and hence $H(T_n)<\infty$.

It remains to show that $H(T_n)$ is large. For this, we use the following simple observation: if a finite packing by copies of a shape $T$, where the projection of each copy onto any coordinate axis is of length at most $D$, covers a cube whose edges are parallel to the coordinate axes and longer than $2hD$,  then $H(T)\geqslant h$. (Indeed, starting from a member of the packing that contains the center of the cube, let each corona consist of the members not used so far that share a point with some member of the previous corona; as each corona reaches at most $D$ further from the center than the previous one along each coordinate direction, the starting member and the first $h-1$ coronas lie in the interior of the cube, so the total number of coronas is at least $h$.)

We are now ready to prove the main result.

\begin{theorem}\label{main}
For every positive integer $h$, there exists a shape $T$ in $\E^3$ such that $h\leqslant H(T)<\infty$.
\end{theorem}

\begin{proof}
Let $n=3h+4$ and $T=T_n$. We have already seen that  $H(T)<\infty$. The packing from Lemma \ref{p41} (with copies of $T$ instead of copies of $K_n$) covers a cube with edge length more than $4n^2-12n-2$ (recall that bumps and nicks are small). Moreover, by (c), the projection of each member of the new packing on any coordinate axis is of length at most $6n+1$. We have $4n^2-12n-2=36h^2+60h+14$ and $2h(6n+1)=36h^2+50h$, so the edge length of the covered cube is greater than $2h(6n+1)$. Therefore, $H(T)\geqslant h$, which was to be proved.
\end{proof}

\section*{Acknowledgments}

The author was supported by the Ministry of Science, Technological Development and Innovation of the Republic of Serbia (grants no. 451-03-33/2026-03/200125 and 451-03-34/2026-03/200125).


\begin{thebibliography}{xx}
\bibitem{ammann}R. Ammann, in: Special session on tilings, 868th meeting of the American Mathematical Society, Philadelphia, Pennsylvania, October 12--13, 1991.
\bibitem{basic15}B. Ba\v si\'c, The Heesch number for multiple prototiles is unbounded, \emph{C. R. Math. Acad. Sci. Paris} \textbf{353} (2015), 665--669.
\bibitem{basic21}B. Ba\v si\'c, A figure with Heesch number $6$: pushing a two-decade-old boundary, \emph{Math. Intelligencer} \textbf{43} (2021), 50--53.
\bibitem{bds}B. Ba\v si\'c \& A. D\v zuklevski \& A. Slivkov\'a, Solutions to seven and a half problems on tilings, \emph{Electron. J. Combin.} \textbf{30} (2023), Article No. 2.50, 35 pp.
\bibitem{bs22}B. Ba\v si\'c \& A. Slivkov\'a, Asymptotical unboundedness of the Heesch number in $\mathbb E^d$ for $d\to\infty$, \emph{Discrete Comput. Geom.} \textbf{67} (2022), 328--337.
\bibitem{bs25}B. Ba\v si\'c \& A. Slivkov\'a, The existence of hypersolids in $\mathbb E^d$ whose Heesch number is $d-1$, \emph{Contrib. Discrete Math.} \textbf{20(1)} (2025), 105--116.
\bibitem{berger}R. Berger, The undecidability of the domino problem, \emph{Mem. Amer. Math. Soc.} \textbf{66} (1966), 72 pp.
\bibitem{danzer}L. Danzer, A family of 3D-spacefillers not permitting any periodic or quasiperiodic tiling, in: G. Chapuis \& W. Paciorek (eds.), \emph{Aperiodic '94}, World Scientific, Singapore, 1995, pp. 11--17.
\bibitem{dl}E. D. Demaine \& S. Langerman, Tiling 3D by translates of a single polycube is undecidable, preprint, arXiv:2610.12392.
\bibitem{fontaine}A. Fontaine, An infinite number of plane figures with Heesch number two, \emph{J. Combin. Theory Ser. A} \textbf{57} (1991), 151--156.
\bibitem{gt}R. Greenfeld \& T. Tao, Undecidability of translational monotilings, \emph{J. Eur. Math. Soc.}, published online 2025, doi:10.4171/JEMS/1673.
\bibitem{gs}B. Gr\"unbaum \& G. C. Shephard, \emph{Tilings and Patterns}, W. H. Freeman, New York, 1987.
\bibitem{heesch35}H. Heesch, Aufbau der Ebene aus kongruenten Bereichen, \emph{Nachr. Ges. Wiss. G\"ottingen, Math.-Phys. Kl., Fachgr. I (N.F.)} \textbf{1} (1935), 115--117.
\bibitem{heesch68}H. Heesch, \emph{Regul\"ares Parkettierungsproblem}, Westdeutscher Verlag, K\"oln--Opladen, 1968.
\bibitem{hilbert}D. Hilbert, Mathematical problems, \emph{Bull. Amer. Math. Soc.} \textbf{8} (1902), 437--479.
\bibitem{kaplan}C. S. Kaplan, Heesch numbers of unmarked polyforms, \emph{Contrib. Discrete Math.} \textbf{17} (2022), 150--171.
\bibitem{lietzmann}W. Lietzmann, \emph{Lustiges und Merkw\"urdiges von Zahlen und Formen,} Hirt, Breslau, 1928.
\bibitem{maiti}A. Maiti, Unboundedness of the Heesch number for hyperbolic convex monotiles, preprint, arXiv:2603.27827.
\bibitem{mann04}C. Mann, Heesch's tiling problem, \emph{Amer. Math. Monthly} \textbf{111} (2004), 509--517.
\bibitem{mann-thomas}C. Mann \& B. C. Thomas, Heesch numbers of edge-marked polyforms, \emph{Experiment. Math.} \textbf{25} (2016), 281--294.
\bibitem{newman}D. J. Newman, Tesselation of integers, \emph{J. Number Theory} \textbf{9} (1977), 107--111.
\bibitem{reinhardt}K. Reinhardt, Zur Zerlegung der euklidischen R\"aume in kongruente Polytope, \emph{Sitzungsber. Preuss. Akad. Wiss. Berlin} (1928), 150--155.
\bibitem{robinson}R. M. Robinson, Undecidability and nonperiodicity for tilings of the plane, \emph{Invent. Math.} \textbf{12} (1971), 177--209.
\bibitem{smkg1}D. Smith \& J. S. Myers \& C. S. Kaplan \& C. Goodman-Strauss, An aperiodic monotile, \emph{Combin. Theory} \textbf{4(1)} (2024), Article No. 6, 91 pp.
\bibitem{smkg2}D. Smith \& J. S. Myers \& C. S. Kaplan \& C. Goodman-Strauss, A chiral aperiodic monotile, \emph{Combin. Theory} \textbf{4(2)} (2024), Article No. 13, 25 pp.
\bibitem{tarasov}A. S. Tarasov, On the Heesch number for the hyperbolic plane, \emph{Math. Notes} \textbf{88} (2010), 97--102.
\bibitem{wang}H. Wang, Proving theorems by pattern recognition II, \emph{Bell System Tech. J.} \textbf{40} (1961), 1--41.
\end{thebibliography}
\end{document}